\documentclass[global,final]{svjour}
\date{29 February, 2008}
\usepackage{amsmath,amsfonts,amssymb}
\usepackage[dvips]{graphicx}
\usepackage[cmtip,all]{xy}
\usepackage[numbers,square,sort&compress]{natbib} 
\usepackage{xspace}
\graphicspath{{./}{./figs/}{../../figs/}}
\def\F{\mathbb{F}}
\def\N{\mathbb{N}}
\def\Z{\mathbb{Z}}
\def\Fy{\mathfrak{F}_Y}
\def\SYmod{S_Y\!/\!\!\sim_Y}
\def\simkappa{\!\!\sim_{\kappa}}
\def\tsimkappa{\sim_{\kappa}}

\def\tsimdelta{\sim_{\delta}}
\def\sds{{\sf SDS}\xspace}
\def\sdss{{\sf SDS}s\xspace}
\def\Acyc{\mathsf{Acyc}}
\def\acyc{\alpha}
\def\Shift{\ensuremath{\text{\boldmath{$\sigma$}}}}
\def\rev{\ensuremath{\text{\boldmath{$\rho$}}}}
\def\<{\langle}
\def\>{\rangle}
\def\card#1{|#1|}

\def\vset{\mathrm{v}}
\def\eset{\mathrm{e}}
\def\eg{e.g.\xspace}
\def\ie{i.e.\xspace}
\def\nkeq{\nsim_\kappa}
\def\varepsilon{\eta_e}
\def\ieext{\mathcal{I}^*_e}
\begin{document}
\title{Equivalences on Acyclic Orientations}
\author{
Matthew~Macauley\inst{1,3} \and
Henning~S.~Mortveit\inst{2,3}}
\institute{
Department of Mathematics, UCSB \and 
Department of Mathematics, Virginia Tech \and 
NDSSL, VBI, Virginia Tech, \email{\{macauley,henning.mortveit\}@vt.edu} 
\thanks{This work was partially supported by Fields Institute in
  Toronto, Canada. 2000 Mathematics Subject Classification:
  06A06;05A99;05C20;20F55} 
}

\maketitle

\keywords{
Acyclic orientations, partially ordered sets, equivalence, permutations,
shift, reflection, Tutte polynomial.
}

\medskip

\abstract{
  The cyclic and dihedral groups can be made to act on the set
  $\Acyc(Y)$ of acyclic orientations of an undirected graph $Y$, and
  this gives rise to the equivalence relations $\tsimkappa$ and
  $\tsimdelta$, respectively. These two actions and their
  corresponding equivalence classes are closely related to
  combinatorial problems arising in the context of Coxeter groups,
  sequential dynamical systems, the chip-firing game, and
  representations of quivers.

  In this paper we construct the graphs $C(Y)$ and $D(Y)$ with vertex
  sets $\Acyc(Y)$ and whose connected components encode the
  equivalence classes. The number of connected components of these
  graphs are denoted $\kappa(Y)$ and $\delta(Y)$, respectively. We
  characterize the structure of $C(Y)$ and $D(Y)$, show how
  $\delta(Y)$ can be derived from $\kappa(Y)$, and give enumeration
  results for $\kappa(Y)$. Moreover, we show how to associate a poset
  structure to each $\kappa$-equivalence class, and we characterize
  these posets. This allows us to create a bijection from
  $\Acyc(Y)/\simkappa$ to $\Acyc(Y')/\simkappa \cup\,\,
  \Acyc(Y'')/\simkappa$, where $Y'$ and $Y''$ denote edge deletion
  and edge contraction for a cycle-edge in $Y$, respectively, which in
  turn shows that $\kappa(Y)$ may be obtained by an evaluation of the
  Tutte polynomial at $(1,0)$.
}


\section{Introduction}
\label{sec:intro}
An acyclic orientation $O_Y$ of an undirected graph $Y$ induces a
partial ordering on the vertex set $\vset[Y]$ by $i\leq_{O_Y} j$ if
there is a directed path from $i$ to $j$ in $O_Y$.
A cyclic $1$-shift (left) of a linear extension of $O_Y$ corresponds
to converting a source of $O_Y$ into a sink. This source-to-sink
operation on $\Acyc(Y)$ gives rise to the equivalence relation
denoted $\tsimkappa$. Reversing a linear extension of $O_Y$
corresponds to reflecting all edge orientations in $O_Y$. The coarser
equivalence relation on $\Acyc(Y)$ obtained through source-to-sink
operations and reflections is denoted $\tsimdelta$.

This paper is organized as follows. In Section~\ref{sec:equivalence}
we construct the equivalence relations $\tsimkappa$ and $\tsimdelta$,
and graphs $C(Y)$ and $D(Y)$ that have vertex set $\Acyc(Y)$, and
whose connected components corresponds to $\kappa$- and
$\delta$-equivalence classes, respectively. We let $\kappa(Y)$ and
$\delta(Y)$ denote the number of equivalence classes. In
Section~\ref{sec:structure} we study the structure and properties of
the graphs $C(Y)$ and $D(Y)$, show how $\delta(Y)$ can be determined
from $\kappa(Y)$, and give bounds for these quantities. In
Section~\ref{sec:posets} we show that one may associate a poset to each
$\kappa$-equivalence class and use this to establish the bijection
\begin{equation}
  \label{eq:bijection1}
  \Theta\colon\Acyc(Y)/\simkappa\,\longrightarrow
  \big(\Acyc(Y'_e)/\simkappa\!\big)\,\bigcup\,
  \big(\Acyc(Y''_e)/\simkappa\!\big)\;,
\end{equation}
where $Y'_e$ and $Y''_e$ are the graphs formed by deleting and
contracting a non-bridge edge $e$ of $Y$, respectively. This leads to
a new proof of a recursion relation for $\kappa(Y)$
in~\cite{Macauley:08b}. Finally, in the summary, we discuss how these
constructions arise in other areas of mathematics, such as sequential
dynamical systems, Coxeter groups, the chip-firing game, and the
representation theory of quivers.


\section{Terminology and Background}
\label{sec:prelim}

Let $Y$ be an undirected, simple and loop-free graph with vertex set
$\vset[Y]=\{1,2,\dots,n\}$ and edge set $\eset[Y]$. We let $S_Y$
denote the set of total orders (\ie, permutations) of
$\vset[Y]$. In~\cite{Reidys:98a} an equivalence relation $\sim_Y$ is
introduced on $S_Y$ through the \emph{update graph} $U(Y)$ of $Y$.
The update graph has vertex set $S_Y$, and two distinct vertices $\pi
= (\pi_i)_i$ and $\pi' = (\pi'_i)_i$ are adjacent if they differ in
exactly two consecutive elements $\pi_k$ and $\pi_{k+1}$ such that
$\{\pi_k, \pi_{k+1}\} \not\in\eset[Y]$. The equivalence relation
$\sim_Y$ is defined by $\pi \sim_Y \pi'$ if $\pi$ and $\pi'$ are
connected in $U(Y)$. We denote the equivalence class containing $\pi$
by $[\pi]_Y$, and set
\begin{equation*}
  \SYmod\, = \big\{[\pi]_Y\mid\pi\in S_Y\big\}\;.
\end{equation*}
This corresponds to partially commutative monoids as defined
in~\cite{Cartier:69}, but restricted to fixed length permutations over
$\vset[Y]$ and with commutation relations encoded by non-adjacency in
the graph $Y$.

Orientations of $Y$ are represented as maps $O_Y\colon\eset[Y]
\longrightarrow \vset[Y] \times \vset[Y]$, which may also be viewed as
directed graphs. The set of acyclic orientations of $Y$ is denoted
$\Acyc(Y)$, and we set $\acyc(Y) = \card{\Acyc(Y)}$. 
Every acyclic orientation defines a partial ordering on $\vset[Y]$
where the covering relations are $i\leq_{O_Y}\!j$ if $\{i,j\}\in
\eset[Y]$ and $O_Y(\{i,j\}) = (i,j)$.
The set of linear extensions of $O_Y$ contains precisely the
permutations $\pi\in S_Y$ such that if $i\leq_{O_Y} j$, then $i$
precedes $j$ in $\pi$. Through the ordering of $\vset[Y]$, every
permutation $\pi\in S_Y$ induces a canonical linear order on
$\vset[Y]$, see~\cite{Reidys:98a}. Moreover, each permutation $\pi\in S_Y$
induces an acyclic orientation $O_Y^\pi\in\Acyc(Y)$ defined by
$O_Y^{\pi}(\{i,j\}) = (i,j)$ if $i$ precedes $j$ in $\pi$ and
$O_Y^{\pi}(\{i,j\}) = (j,i)$ otherwise.
The bijection
\begin{equation}
  \label{eq:bij1}
  f_Y\colon\SYmod\,\longrightarrow \Acyc(Y)\;,\qquad 
     f_Y([\pi]_Y) = O_Y^\pi\;,
\end{equation}
from~\cite{Reidys:98a} allows us to identify equivalence classes 
and acyclic orientations. The number of equivalence classes under
$\sim_Y$ is therefore given by $\acyc(Y)$. 

For $O_Y\in\Acyc(Y)$ and $e = \{v,w\}\in\eset[Y]$ let $O^{\rho(e)}_Y$
be the orientation of $Y$ obtained from $O_Y$ by reversing the
edge-orientation of $e$. Let $Y'_e$ and $Y''_e$ denote the graphs
obtained from $Y$ by deletion and contraction of $e$, respectively,
and let $O_{Y'}$ and $O_{Y''}$ denote the induced orientations of $O_Y$
under these operations. The bijection
\begin{equation}
  \label{eq:bijection2}
  \beta_e \colon \Acyc(Y)\longrightarrow\Acyc(Y_e')\cup\Acyc(Y''_e) 
\end{equation}
defined by
\begin{eqnarray*}
  O_Y&\longmapsto&
  \begin{cases}
    O'_Y\;,  & O^{\rho(e)}_Y\not\in\Acyc(Y),  \\ 
    O'_Y\;,  & O^{\rho(e)}_Y\in\Acyc(Y)\mbox{ and } O_Y(e)=(v,w) \;,  \\  
    O''_Y\;, & O^{\rho(e)}_Y\in\Acyc(Y)\mbox{ and } O_Y(e)=(w,v) \;.  \\  
  \end{cases}
\end{eqnarray*}
is well-known, and shows that one may compute $\alpha(Y)$ through the
recursion relation
\begin{equation*}
  \label{eq:alpha}
  \acyc(Y)=\acyc(Y'_e)+\acyc(Y''_e) \;,
\end{equation*}
valid for any $e\in\eset[Y]$. 


\section{Graph Constructions for Equivalence Relations}
\label{sec:equivalence}


\subsection{Relations on $\SYmod$}

Using cycle notation, let $\sigma,\rho\in S_n$ be the elements
\begin{equation*}
  \sigma = (n, n-1,\ldots,2,1)\;, \qquad 
  \rho = (1,n)(2,n-1)\cdots (\lceil \tfrac{n}{2}\rceil, 
  \lfloor \tfrac{n}{2} \rfloor +1)\;,
\end{equation*}
and let $C_n$ and $D_n$ be the subgroups
\begin{equation}
  C_n = \<\sigma\> \text{\quad and\quad }
  D_n = \<\sigma,\rho\> \;.
\end{equation}
Both $C_n$ and $D_n$ act on $S_Y$ via $g\cdot(\pi_1,\ldots,\pi_n) =
(\pi_{g^{-1}(1)},\ldots,\pi_{g^{-1}(n)})$.  Define $\Shift_s(\pi) =
\sigma^s\cdot\pi$, so that, \eg $\Shift_1(\pi) = \sigma\cdot\pi =
(\pi_2,\pi_3,\ldots,\pi_n,\pi_1)$, and define $\rev(\pi) =
\rho\cdot\pi=(\pi_n,\pi_{n-1},\ldots,\pi_2,\pi_1)$. We construct two
undirected graphs $C(Y)$ and $D(Y)$ whose vertex sets are $\SYmod$,
and edge sets are
\begin{eqnarray*}
\eset[C(Y)]&=&\bigl\{\{[\pi]_Y,[\Shift_1(\pi)]_Y\} \mid \pi\in
  S_Y\bigr\}\;,\\ 
  \eset[D(Y)]&=&  \bigl\{\{[\pi]_Y,\,[\rev(\pi)]_Y\}\mid \pi\in S_Y\bigr\}
  \cup\eset[C(Y)] \;.
\end{eqnarray*}
Define $\kappa(Y)$ and $\delta(Y)$ to be the number of connected
components of $C(Y)$ and $D(Y)$, respectively. By construction, $C(Y)$
is a subgraph of $D(Y)$ and $\delta(Y) \leq \kappa(Y)$.
\begin{example}
  \label{ex:K23}
  Let $Y$ be the complete bipartite graph $K_{2,3}$, where the
  partition of the vertex set is $\{\{1,3,5\},\{2,4\}\}$. The graph
  $U(K_{2,3})$ is shown in Figure~\ref{fig:K23} with vertex labels
  omitted.  By simply counting the components we see that
  $\alpha(K_{2,3}) = 46$.
  \begin{figure}[ht]
    \centerline{\includegraphics[width=0.95\textwidth]{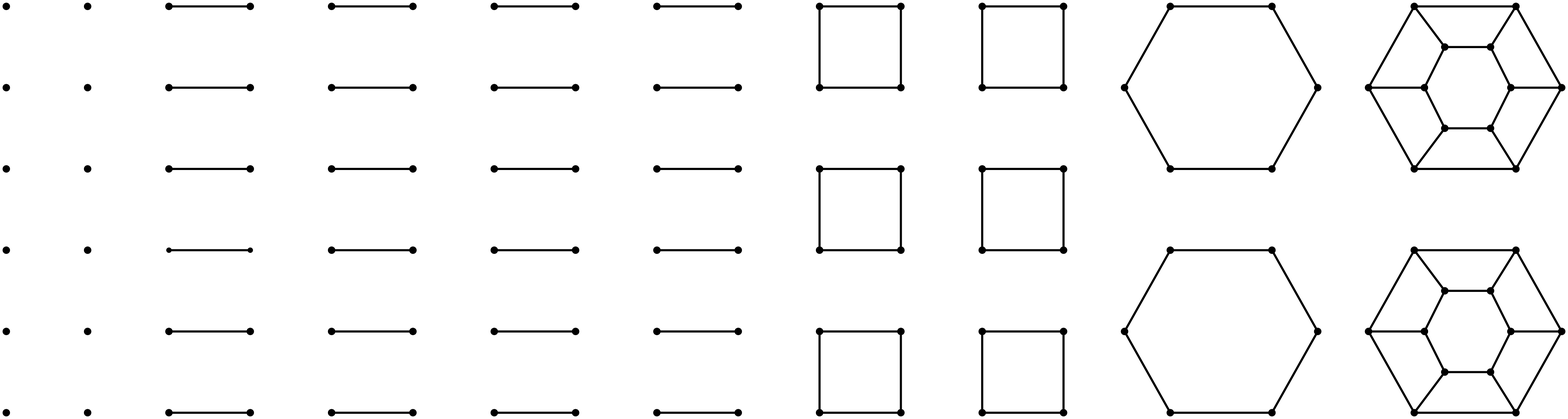}}
    \caption{The update graph $U(K_{2,3})$.}
    \label{fig:K23} 
  \end{figure}
  We can better understand the component structure of $U(K_{2,3})$ by
  mapping permutations as $(\pi_i)_i \stackrel{\phi}{\mapsto} (\pi_i
  \mod 2)_i$. Non-adjacency in $Y$ coincides with parity, that is, if
  $\pi\sim_Y\sigma$, then $\phi(\pi) = \phi(\sigma)$. Through the map
  $\phi$ we see that the $12$ singleton points in $U(K_{2,3})$ are
  precisely those with image $10101$. Each of the $24$ size-two
  components correspond to a pair of permutations with $\phi$-image of
  the form $01011$, $11010$, $01101$, or $10110$. The six
  square-components arise from the permutations with $\phi$-image
  $10011$ and $11001$. Finally, the permutations in the two
  hexagon-components are of the form $01110$, and those in the two
  largest components have $\phi$-image of the form $11100$ or $00111$.
  
  The graphs $C(K_{2,3})$ and $D(K_{2,3})$ are shown in
  Figure~\ref{fig:C_K23}.  The dashed lines are edges that belong to
  $D(K_{2,3})$ but not to $C(K_{2,3})$.
  \begin{figure}[!htbp]
    \centerline{\includegraphics[width=.95\textwidth]{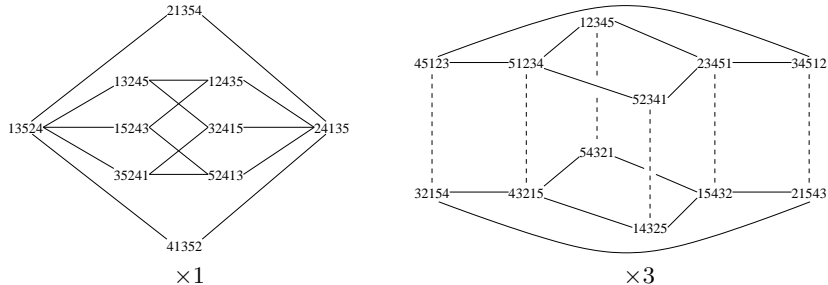}}
    \centerline{\hfill  $\times 1$ \hskip0.475\textwidth $\times 3$ \hfill}
    \caption{The graph $C(K_{2,3})$ contains the component on the left,
      and three isomorphic copies of the structure on the right (but
      with different vertex labels). The dashed lines are edges in
      $D(K_{2,3})$ but not in $C(K_{2,3})$.}
    \label{fig:C_K23}
  \end{figure}
  The vertices in Figure~\ref{fig:C_K23} are labeled by a permutation
  in the corresponding equivalence class in $\SYmod$. There are three
  isomorphic copies of the component on the right, but only one is
  shown. Each of these three components contains permutations whose
  $\phi$-image is in $\{01101, 11010, 10101, 01011, 10110\}$. The
  component on the left contains all of the remaining permutations,
  i.e., all $\pi$ for which $\phi(\pi) \in \{11100, 11001, 10011,
  00111, 01110\}$. Clearly, $\kappa(K_{2,3})=7$ and
  $\delta(K_{2,3})=4$.
\end{example}

Permutations from $\sim_Y$ classes belonging to the same component in
$C(Y)$ are called \emph{$\kappa$-equivalent} permutations, as are the
corresponding acyclic orientations.  For two $\kappa$-equivalent
permutations $\pi$ and $\pi'$ there is a sequence of adjacent non-edge
transpositions and cyclic shifts that map $\pi$ to $\pi'$. This is
simply a consequence of the definitions of $\SYmod$ and $C(Y)$.
Similarly, two permutations belonging to $\sim_Y$ classes on the same
connected component in $D(Y)$ are called \emph{$\delta$-equivalent},
as are their corresponding acyclic orientations.


\subsection{Relations on $\Acyc(Y)$}
\label{subsec:acyclic}

The bijection between $\SYmod$ and $\Acyc(Y)$ in~\eqref{eq:bij1}
allows us to identify $[\pi]_Y$ with the acyclic orientation
$O_Y^\pi$. It is clear that mapping $\pi\in[\pi']_Y$ to
$\Shift_1(\pi)$ corresponds precisely to converting the vertex $\pi_1$
from a source to a sink in $O_Y^\pi$.  This can be
extended. Following~\cite{Shi:01} we call the conversion of a source
vertex to a sink vertex in $O_Y\in\Acyc(Y)$ a \emph{source-to-sink
operation}, or a \emph{click}. Two orientations $O_Y,O_Y'\in \Acyc(Y)$
where $O_Y$ can be transformed into $O'_Y$ by a sequence of clicks are
said to be click-related, and we write this as $\mathbf{c}(O_Y) =
O'_Y$ where $\mathbf{c}=c_1c_2\cdots c_k$ with $c_i\in\vset[Y]$. It is
straightforward to verify that this click-relation is an equivalence
relation on $\Acyc(Y)$.  Through the bijection in~\eqref{eq:bij1} it
is clear that a source-to-sink operation precisely encodes adjacency
in the graph $C(Y)$, and the number of click-equivalence classes in
$\Acyc(Y)$ therefore equals $\kappa(Y)$.
The second equivalence relation on $\Acyc(Y)$ arises in the same
manner by additionally identifying $O_Y^{\pi}$ and the reverse
orientation $O_Y^{\rev(\pi)}$, the unique orientation that satisfies
$O_Y^{\pi}(\{i,j\})\neq O_Y^{\rev(\pi)}(\{i,j\})$ for every
$\{i,j\}\in\eset[Y]$.


\section{Structure of $C(Y)$ and $D(Y)$}
\label{sec:structure}

The following result gives insight into the component structure of the
graph $C(Y)$.

\begin{proposition}
  \label{prop:C_n}
  Let $Y$ be a connected graph on $n$ vertices and let $g,g'\in C_n$
  with $g\ne g'$. Then $[g\cdot\pi]_Y \ne [g'\cdot\pi]_Y$.
\end{proposition}
\begin{proof}
  Assume $g\neq g'$ with $[g \cdot\pi]_Y=[g'\cdot\pi]_Y$.  By
  construction, we have $g \cdot\pi=\Shift_s(\pi)$ and $g'\cdot\pi =
  \Shift_{s'}(\pi)$.  Without loss of generality we may
  assume~$s'<s$. Let $V'\subset V=\vset[Y]$ be the initial subsequence
  of vertices in $\Shift_{s'}(\pi)$ that occurs at the end in
  $\Shift_s(\pi)$. If any of the vertices in $V'$ are adjacent to any
  of the vertices in $V \setminus V'$ in $Y$ it would imply that
  $[\Shift_s(\pi)]_Y \neq [\Shift_{s'}(\pi)]_Y$. The only possibility
  is that $Y$ is not connected, but this contradicts the assumptions
  of the proposition.  \qed
\end{proof}

There is a similar result to Proposition~\ref{prop:C_n} for $D_n$,
albeit somewhat more restrictive.
\begin{proposition}
  \label{prop:D_n}
  Let $Y$ be a connected graph on $n$ vertices and let $g,g'\in D_n$
  with $g \ne g'$. If $[g\cdot\pi]_Y=[g'\cdot\pi]_Y$ holds then $Y$
  must be bipartite.
\end{proposition}

\begin{proof}
  From $[g\cdot\pi]_Y=[g'\cdot\pi]_Y$ it follows from
  Proposition~\ref{prop:C_n} that $g$ and $g'$ lie in different cosets
  of $C_n$ in $D_n$. Without loss of generality we may assume that
  $g=\sigma^s$ and $g'=\rho\sigma^{s'}$. Let $m=\card{s'-s}$ and $m' =
  n - m$. If $s'>s$ (resp. $s'<s$) the first (resp. last) $m$ elements
  of $g\cdot\pi$ and $g'\cdot \pi$ are the same but occur in reverse
  order. Call the set of these elements $V_1$. The remaining $m'$
  elements occur in reverse order as well in the two permutations. Let
  $V_2$ denote the set of these elements. For
  $[g\cdot\pi]_Y=[g'\cdot\pi]_Y$ to hold, there cannot be an edge
  between any two vertices in $V_1$, or between any two vertices in
  $V_2$. Therefore, the graph $Y$ must be a subgraph of $K(V_1,V_2)$,
  the complete bipartite graph with vertex sets $V_1$ and $V_2$.  \qed
\end{proof}

\begin{remark}
  \label{rem:bip1}
  The pairs $(\sigma^s, \rho \sigma^{s'})$ and $(\sigma^{s'},
  \rho\sigma^{s})$ determine the same bipartite graph in the proof
  above. Also, the vertex sets $V_1$ and $V_2$ can only consist of
  consecutive elements in~$\pi$.
\end{remark}
\begin{remark}
  \label{rem:bip2}
  If $Y$ is connected and bipartite then $\card{\{[g\cdot\pi]_Y \mid
    g\in D_n\}} = 2n-1$. This follows from the fact that at most two
  $\sim_Y$ classes can coincide as all distinct pairs $g$ and $g'$ for
  which equality holds leads to different sets $V_1(\{g,g'\})$ and
  $V_2(\{g,g'\})$ modulo Remark~\ref{rem:bip1}. The existence of two
  or more distinct partitions of $\vset[Y]$ into sets $V_1$ and $V_2$
  as above would imply that $Y$ is not connected.
\end{remark}

We next consider the quantity $\delta(Y)$, and will show how it is
determined by $\kappa(Y)$.

\begin{lemma}
  \label{lem:rho}
  The reflection map $\rev \colon S_Y \longrightarrow S_Y$ extends to an
  involution
  \begin{equation}
    \label{eq:rho-star}
    \rev^* \colon \Acyc(Y)/\simkappa \longrightarrow 
    \Acyc(Y)/\simkappa \;.
  \end{equation}
\end{lemma}
\begin{proof}
  By the definition of $U(Y)$ it follows that if $\pi,\pi'\in S_Y$ are
  adjacent in $U(Y)$ then so are $\rev(\pi)$ and $\rev(\pi')$. It
  follows easily that $\pi\sim_Y\pi'$ implies
  $\rev(\pi)\sim_Y\rev(\pi')$. The map $\rev$ therefore extends to a
  map $\hat{\rev} \colon \SYmod \longrightarrow \SYmod$ by
  $\rev([\pi]_Y) = [\rev(\pi)]_Y$. Likewise, if $O_Y$ and $O_Y'$ are
  $\kappa$-equivalent then so are $\hat{\rev}(O_Y)$ and
  $\hat{\rev}(O_Y')$ [using the bijection~\eqref{eq:bij1}], and
  $\hat{\rev}$ extends to $\rev^*$ as in~\eqref{eq:rho-star} by
  $\rev^*(A) = \hat{\rev}(O_A)$ for any $O_A \in A \in
  \Acyc(Y)/\simkappa$. This map is clearly an involution since $\rev$
  itself is an involution. \qed
\end{proof}

\begin{proposition}
  \label{prop:delta}
  Let $Y$ be a connected undirected graph. If $Y$ is not bipartite then
  $\delta(Y)=\tfrac{1}{2}\kappa(Y)$. If $Y$ is bipartite
  then $\delta(Y)=\tfrac{1}{2}(\kappa(Y)+1)$.
\end{proposition}

\begin{proof}
  If $Y$ is not bipartite then Proposition~\ref{prop:D_n} ensures that
  the involution $\rho^*$ has no fixed points, from which we conclude
  that $2\delta(Y) = \kappa(Y)$.

  For the second statement, we use Proposition~\ref{lem:rho},
  Remark~\ref{rem:bip2}, and the connectedness of $Y$ to conclude that
  $\rev^*$ has precisely one fixed point. Since $\rev$ is an
  involution it follows that $\delta(Y) = \tfrac{\kappa(Y)-1}{2} + 1 =
  \tfrac{\kappa(Y)+1}{2} = \lceil \kappa(Y)/2 \rceil$.
\qed
\end{proof}
Thus, we always have $\delta(Y) = \lceil \kappa(Y)/2 \rceil$, and we
also have the following characterization of bipartite graphs:
\begin{corollary}
  \label{cor:bipartite}
  A connected graph $Y$ is bipartite if and only if $\kappa(Y)$ is odd.
\end{corollary} 
For an example where $\rev^*$ has a fixed point see
Figure~\ref{fig:C_K23} in Example~\ref{ex:K23}.  From
Proposition~\ref{prop:C_n} we can derive an upper bound for
$\kappa(Y)$.

\begin{proposition}
  \label{prop:1/n}
  If $Y$ is a connected undirected graph on $n$ vertices, then
  $\kappa(Y)\leq\tfrac{1}{n}\,\acyc(Y)$. 
\end{proposition}

\begin{proof}
  By Proposition~\ref{prop:C_n}, for any $\pi\in S_Y$, the set
  $\{O_Y^{\Shift_1(\pi)},\dots,O_Y^{\Shift_n(\pi)}\}$ contains $n$
  distinct acyclic orientations that are all $\kappa$-equivalent and
  the proof follows.
\qed
\end{proof}
This bound is sharp for certain graphs such as the complete graph $K_n$. 


\section{Poset Structure of $\kappa$-Equivalence Classes}
\label{sec:posets}

An edge $e$ of an undirected graph $Y$ is a \emph{bridge} if removing
$e$ increases the number of connected components of $Y$. An edge that
is not a bridge is a \emph{cycle-edge}, or equivalently, an edge $e$ is a
cycle-edge if it is contained in a cycle traversing $e$ precisely
once. The main result in~\cite{Macauley:08b} is a recurrence relation
for $\kappa(Y)$ under edge deletion and edge contraction.
\begin{theorem}[\cite{Macauley:08b}]
  \label{thm:kappa}
  Let $Y$ be a finite undirected graph with $e\in\eset[Y]$, and let
  $Y_e'$ be the graph obtained from $Y$ by deleting $e$, and let
  $Y_e''$ be the graph obtained from $Y$ by contracting $e$. Then
  \begin{equation}
    \label{eq:kappa}
    \kappa(Y) = 
      \left\{
      \begin{array}{cl}
        \kappa(Y_1)\kappa(Y_2), &\;\; 
          \mbox{$e$ is a bridge linking $Y_1$ and $Y_2$}\,, \\
        \kappa(Y_e')+\kappa(Y_e''), &\;\; 
          \mbox{$e$ is a cycle-edge}\;.
      \end{array}
      \right. 
  \end{equation}
\end{theorem}

The first case implies that $\kappa(Y)$ is unaffected by removal of
bridges, and is relatively straightforward to establish. However, the
case where $e$ is a cycle-edge is harder, and was proven
in~\cite{Macauley:08b} on the level of acyclic orientations. In this
section, we will show how one may associate a poset with each
$\kappa$-equivalence class. The properties of this poset give us
better insight into the structure of
$\Acyc(Y)/\simkappa$. Additionally, it allows us to construct an
alternative proof for Theorem~\ref{thm:kappa}.

Throughout, we will let $e=\{v,w\}$ be a fixed cycle-edge of the
connected graph $Y$, and for ease of notation we set $Y'=Y'_e$ and
$Y''=Y''_e$. For $O_Y\in\Acyc(Y)$ we let $O_{Y'}$ and $O_{Y''}$ denote
the induced orientations of $Y'$ and $Y''$. Notice that $O_{Y'}$ is
always acyclic, while $O_{Y''}$ is acyclic if and only if there is no
directed path with endpoints $v$ and $w$ in $O_{Y'}$. Finally, we let
$[O_Y]$ denote the $\kappa$-equivalence class containing $O_Y$.

The interval $[a,b]$ of a poset $\mathcal{P}$ (where $a\leq b$) is the
subposet consisting of all $c\in \mathcal{P}$ such that $a\leq c\leq
b$. Viewing a finite poset $\mathcal{P}$ as a directed graph
$D_{\mathcal{P}}$, the interval $[a,b]$ contains precisely the
vertices that lie on a directed path from $a$ to $b$, and thus is a
vertex-induced subgraph of $D_{\mathcal{P}}$. By assumption, $Y$
contains the edge $\{v,w\}$, so for all $O_Y\in\Acyc(Y)$ either
$v\leq_{O_Y}\!w$ or $w\leq_{O_Y}\!v$. In this section, we will study
the interval $[v,w]$ in the poset $O_Y$ (when $v\leq_{O_Y} w$) and its
behavior under clicks.

\begin{definition}
  Let $\Acyc_\le(Y)$ be the set of acyclic orientations of
  vertex-induced subgraphs of $Y$. We define the map
  \begin{equation*}
    \mathcal{I}\colon\Acyc(Y)\longrightarrow\Acyc_\le(Y)\;,
  \end{equation*}
  by $\mathcal{I}(O_Y) = [v,w]$ if $v\leq_{O_Y}\!w$, and by $\mathcal{I}(O_Y)
  = \varnothing$ (the null graph) otherwise. When
  $\mathcal{I}(O_Y) \ne \varnothing$ we refer to $\mathcal{I}(O_Y)$ as the
  \emph{$vw$-interval} of $O_Y$.
\end{definition}
Elements of $\Acyc_\le(Y)$ can be thought of as subposets of
$\Acyc(Y)$. Through a slight abuse of notation, we will at times refer
to $\mathcal{I}(O_Y)$ as a poset, a directed graph, or a subset of
$\vset[O_Y]$. In this last case, it is understood that the relations
are inherited from $O_Y$.

\medskip

For an undirected path $P = v_1,v_2,\dots,v_k$ in $Y$, we define the
function 
\begin{equation*}
\nu_P\colon\Acyc(Y)\longrightarrow\Z \;, 
\end{equation*}
where $\nu_P(O_Y)$ is the number of edges oriented as $(v_i,v_{i+1})$
in $O_Y$, minus the number of edges oriented as $(v_{i+1},v_i)$. It is
clear that if $P$ is a cycle, then $\nu_P$ is preserved under clicks,
and in this case, $\nu_P$ extends to a map $\nu_P^* \colon \Acyc(Y)/\simkappa
\longrightarrow \Z$.

\medskip

We will now prove a series of structural results about the
$vw$-interval. Since $\{v,w\}\in\eset[Y]$, every $\kappa$-class
contains at least one orientation $O_Y$ with $v\leq_{O_Y}\!w$, and
thus there is at least one element $O_Y$ in each $\kappa$-class with
$\mathcal{I}(O_Y) \ne \varnothing$. The next results shows how we can
extend the notion of $vw$-interval from over $\Acyc(Y)$ to
$\Acyc(Y)/\simkappa$.

\begin{proposition}
  \label{prop:I*}
  The map $\mathcal{I}$ can be extended to a map
  \begin{equation*}
    \mathcal{I}^*\colon\Acyc(Y)/\simkappa\longrightarrow\Acyc_\le(Y)
    \text{\quad by\quad} 
    \mathcal{I}^*([O_Y]) = \mathcal{I}(O^1_Y) \;,
  \end{equation*}
  where $O^1_Y$ is any element of $[O_Y]$ for which $\mathcal{I}(O_Y^1)
  \ne \varnothing$.
\end{proposition}

\begin{proof}
  It suffices to prove that $\mathcal{I}^*$ is well-defined. Consider
  $O^1_Y\tsimkappa O^2_Y$ with $v\leq_{O^i_Y}\!w$ for $i=1,2$. To show
  that $\mathcal{I}(O^1_Y) = \mathcal{I}(O^2_Y)$ let $a$ be a vertex
  in $\mathcal{I}(O^1_Y)$.  Then $a$ lies on a directed path $P'$ from
  $v$ to $w$ in $O^1_Y$, say of length $k\geq 2$. Let $P$ be the cycle
  formed by adding vertex $v$ to the end of $P'$. Clearly
  $\nu_P(O_Y^1) = k-1$ since $O^1_Y(e) = (v,w)$.

  By assumption, $O^2_Y\in[O^1_Y]$ with $v\leq_{O^2_Y}\!w$. Since
  $\nu_P$ is constant on $[O^1_Y]$ it follows from $\nu_P(O_Y^1) = k-1
  = \nu_P(O_Y^2)$ that every edge of $P'$ is oriented identically in
  $O^1_Y$ and $O^2_Y$, and hence that every directed path $P'$ in
  $O_Y^1$ is contained in $O_Y^2$ as well. The reverse inclusion
  follows by an identical argument. \qed
\end{proof}
 
In light of Proposition~\ref{prop:I*}, we define the $vw$-interval of
a $\kappa$-class $[O_Y]$ to be $\mathcal{I}^*([O_Y])$. The $vw$-interval
will be central in understanding properties of click-sequences. First,
we will make a simple observation without proof, which also appears
in~\cite{Speyer:07} in the context of admissible sequences in Coxeter
theory.

\begin{proposition}
  \label{prop:alternating}
  Let $O_Y\in\Acyc(Y)$, let $\mathbf{c}=c_1 c_2\cdots c_m$ be an
  associated click-sequence, and consider any directed edge
  $(v_1,v_2)$ in $O_Y$. Then the occurrences of $v_1$ and $v_2$ in
  $\mathbf{c}$ alternate, with $v_1$ appearing first.
\end{proposition}

Because $\{v,w\}\in\eset[Y]$, we can say more about the vertices in
$\mathcal{I}(O_Y)$ that appear between successive instances in $v$ and $w$
in a click-sequence. 

\begin{proposition}
  Let $O_Y\in\Acyc(Y)$, and let $\mathbf{c}=c_1 c_2\cdots c_m$ be an
  associated click-sequence that contain every vertex of
  $\mathcal{I}(O_Y)$ at least once and with $c_1=v$. Then every vertex of
  $\mathcal{I}(O_Y)$ appears in $\mathbf{c}$ before any vertex in
  $\mathcal{I}(O_Y)$ appears twice. 
\end{proposition}

\begin{proof}
  The proof is by contradiction. Assume the statement is false, and
  let $a\in\mathcal{I}(O_Y)$ be the first vertex whose second instance
  in $\mathbf{c}$ occurs before the first instance of some other
  vertex $z \in \mathcal{I}(O_Y)$. If $a\neq v$, then $a$ is not a
  source in $O_Y$, and there exists a directed edge $(a',a)$. By
  Proposition~\ref{prop:alternating}, $a'$ must appear in $\mathbf{c}$
  before the first instance of $a$, but also between the two first
  instances of $a$. This is impossible, because $a$ was chosen to be
  the first vertex appearing twice in $\mathbf{c}$. That only leaves
  $a=v$, and $v$ must appear twice before the first instance of
  $w$. However, this contradicts the statement of
  Proposition~\ref{prop:alternating} because
  $\{v,w\}\in\eset[Y]$. \qed
\end{proof}

The next result shows that for any click-sequence $\mathbf{c}$ that
contains every element in $\mathcal{I}(O_Y)$ precisely once, we may assume
without loss of generality that the vertices in $\mathcal{I}(O_Y)$
appear consecutively.

\begin{proposition}
  \label{prop:forcedclick}
  Let $O_Y\in\Acyc(Y)$ be an acyclic orientation with $v\leq_{O_Y} w$.
  If $\mathbf{c} = c_1c_2\cdots c_m$ is an associated click-sequence
  containing precisely one instance of $w$, and no subsequent
  instances of vertices from $\mathcal{I}(O_Y)$, then there exists a
  click-sequence $\mathbf{c}' = c'_1c'_2\cdots c'_m$ such that $(i)$
  there exists an interval $[p,q]$ of $\N$ with
  $c'_j\in\mathcal{I}(O_Y)$ iff $p\leq j\leq q$, and $(ii)$
  $\mathbf{c}(O_Y) = \mathbf{c}'(O_Y)$.
\end{proposition}

\begin{proof}
  We prove the proposition by constructing a desired click-sequence
  $\mathbf{c}''$ from $\mathbf{c}$ through a series of transpositions
  where each intermediate click-sequence $\mathbf{c}'$ satisfies
  $\mathbf{c}(O_Y) = \mathbf{c}'(O_Y)$. Such transpositions are said
  to have property $T$. 

  Let $I=\mathcal{I}(O_Y)$, and let $A$ be the set of vertices in $I^c
  = \vset[Y]\setminus I$ that lie on a directed path in $O_Y$ to a
  vertex in $I$ (vertices \emph{above} $I$), and let $B$ be the set of
  vertices that lie on a directed path in $O_Y$ from a vertex in $I$
  (vertices \emph{below} $I$). Let $C$ be the complement of $I\cup
  A\cup B$. Two vertices $c_i,c_j\in A\cup B$ with $i<j$ for which
  there is no element $c_k \in A\cup B$ with $i<k<j$ are said to be
  \emph{tight}. We will investigate when transpositions of tight
  vertices in a click-sequence $\mathbf{c}$ of $O_Y$ has property $T$,
  and we will see that this is always the case if $c_i\in B$ and
  $c_j\in A$.
  Consider the intermediate acyclic orientation after applying
  successive clicks $c_1c_2\cdots c_{i-1}$ to $O_Y$. Obviously, $c_i$
  is a source. At this point, if $c_j$ were not a source, then there
  would be an adjacent vertex $a\in A$ with the edge $\{a,c_j\}$
  oriented $(a,c_j)$. For $c_j$ to be clicked as usual (i.e., as a
  source), $a$ must be clicked first, but this would break the
  assumption that $c_i$ and $c_j$ are tight. Therefore, $c_i$ and
  $c_j$ are both sources at this intermediate step, and so the
  vertices $c_i,c_{i+1},\dots,c_j$ are an independent set of sources,
  and may be permuted in any manner without changing the image of the
  click sequence. Therefore, the transposition of $c_i$ and $c_j$ in
  $\mathbf{c}$ has property $T$, as claimed. By iteratively
  transposing tight pairs in $\mathbf{c}$, we can construct a
  click-sequence with the property that every vertex in $A$ comes
  before every vertex in $B$. In light of this, we may assume without
  loss of generality that $\mathbf{c}$ has this property.

  The next step is to show that we can move all vertices in $A$ before
  $v$, and all vertices in $B$ after $w$ via transpositions having
  property $T$.  Let $a$ be the first vertex in $A$ appearing after
  $v$ in the click sequence $\mathbf{c}$.  We claim that the
  transposition moving $a$ to the position directly preceding $v$ has
  property $T$. This is immediate from the observation that when $v$
  is to be clicked, $a$ is a source as well, by the definition of $A$,
  thus it may be clicked before $v$, without preventing subsequent
  clicks of vertices up until the original position of $a$. Therefore,
  we may one-by-one move the vertices in $A$ that are between $v$ and
  $w$, in front of $v$. An analogous argument shows that we may move
  the vertices in $B$ that appear before $w$ to a position directly
  following $w$. In the resulting click-sequence $\mathbf{c}'$, the
  only vertices between $v$ and $w$ are either in $I$ or $C$. The
  subgraph of the directed graph $O_Y$ induced by $C$ is a disjoint
  union of weakly connected components, and none of the vertices are
  adjacent to $I$. By definition of $A$ and $B$, there cannot exist a
  directed edge $(c,a)$ or $(b,c)$, where $a\in A$, $b\in B$, and
  $c\in C$. Thus for each weakly connected component of $C$, the
  vertices in the component can be moved within $\mathbf{c}'$,
  preserving their relative order, to a position either $(i)$ directly
  after the vertices in $A$ and before $v$, or $(ii)$ directly after
  $w$ and before the vertices of $B$. Call this resulting
  click-sequence $\mathbf{c}''$. As we just argued, all the
  transpositions occurring in the rearrangement
  $\mathbf{c}\mapsto\mathbf{c}''$ has property $T$, and
  $\mathbf{c}''$ contains all of the vertices in $I$ in consecutive
  order, and this proves the result. \qed
\end{proof}

We remark that the last two results together imply that for the
interval $[p,q]$ in the statement of
Proposition~\ref{prop:forcedclick}, $c_p=v$, $c_q=w$, and the sequence
$c_pc_{p+1}\cdots c_q$ contains every vertex in $\mathcal{I}(O_Y)$
precisely once. A simple induction argument implies the following.

\begin{corollary}
  \label{cor:intervals}
  Suppose that $O_Y\in\Acyc(Y)$ with $v\leq_{O_Y} w$, and let
  $\mathbf{c}=c_1c_2\cdots c_m$ be a click-sequence where $w$ appears
  $k$ times. Then there exists a click-sequence
  $\mathbf{c}'=c'_1c'_2\cdots c'_m$ such that $(i)$ there are $k$
  disjoint intervals $[p_i,q_i]$ of $\N$ such that
  $c_j\in\mathcal{I}(O_Y)$ iff $p_i\leq j\leq q_i$ for some $i$, and
  $(ii)$ $\mathbf{c}(O_Y)=\mathbf{c}'(O_Y)$.
\end{corollary}

\begin{proof}
  The argument is by induction on $k$. When $k=1$, the statement is
  simply Proposition~\ref{prop:forcedclick}. Suppose the statement holds
  for all $k\leq N$, for some $N\in\N$, and let $\mathbf{c}$ be a
  click-sequence containing $N+1$ instances of $w$. Let $c_\ell$ be
  the second instance of $v$ in $\mathbf{c}$, and consider the two
  click-sequences $\mathbf{c}_i:=c_1c_2\cdots c_{\ell-1}$ and
  $\mathbf{c}_f:=c_\ell c_{\ell+1}\cdots c_m$. By
  Proposition~\ref{prop:forcedclick}, there exists an interval
  $[p_1,q_1]$ with $p_1<q_1<\ell$, and by the induction hypothesis, there
  exists $k$ intervals $[p_2,q_2],\dots,[p_{k+1},q_{k+1}]$ with
  $\ell\leq p_2<q_2<\cdots<p_{k+1}<q_{k+1}$ such that if
  $c_j\in\mathcal{I}(O_Y)$, then $p_i\leq j\leq q_i$ for some
  $i=1,\dots,k+1$. \qed
\end{proof}

Let $\varepsilon\colon\Acyc(Y)\longrightarrow\Acyc(Y')$ be the
restriction map that sends $O_Y$ to $O_{Y'}$. Clearly, this map
extends to a map $\varepsilon^*\colon \Acyc(Y)/\simkappa
\longrightarrow \Acyc(Y')/\simkappa$.
Define 
\begin{equation*}
  \ieext \colon \Acyc(Y')/\simkappa \longrightarrow\Acyc_\le(Y)
\end{equation*}
by $\ieext([O_{Y'}]) = \mathcal{I}(O^1_Y)$ for any $O^1_Y \in [O_Y]$ such
that $\varepsilon^*([O_Y]) = [O_{Y'}]$ with $\card{\mathcal{I}(O^1_Y)}\geq
3$, and $\ieext([O_{Y'}])=\{v,w\}$ if no such acyclic orientation
$O^1_Y$ exists.

\begin{proposition}
  \label{prop:diagram}
  The map $\ieext$ is well-defined, and the diagram
  \begin{equation*}
    \xymatrix{
      \Acyc(Y)/\simkappa
      \ar@{->}[rr]^{\mathcal{I}^*}\ar@{->}[d]_{\varepsilon^*}
      && \Acyc_\le(Y) \\ 
      \Acyc(Y')/\simkappa\ar@{-->}[urr]_{\ieext}
    }
  \end{equation*}
  commutes.
\end{proposition}

\begin{proof}
  Let $[O_{Y'}] \in\Acyc(Y')/\simkappa$. If there is at most one
  orientation $O_{Y}\in\Acyc(Y)$ such that
  $\card{\mathcal{I}(O_Y)}\geq 3$ and $\varepsilon(O_Y) \in [O_{Y'}]$,
  or if all orientations of the form $O_Y^1$ in the definition of
  $\ieext$ are $\kappa$-equivalent, then both statements of the
  proposition are clear. Assume therefore that there are acyclic
  orientations $O_Y^{\pi}, O_Y^{\sigma} \in \Acyc(Y)$ with
  $O_Y^\pi\nkeq O_Y^\sigma$, but
  $\eta^*_e([O_Y^\pi])=\eta^*_e([O_Y^\sigma])$ and
  $\card{\mathcal{I}(O_Y^\pi)},\card{\mathcal{I}(O_Y^\sigma)}\ge 3$.
  It suffices to prove that in this case,
  \begin{equation}
    \label{eq:I_pi=I_sigma}
    \mathcal{I}(O^\pi_Y)=\mathcal{I}(O^\sigma_Y)\;.
  \end{equation}
  This is equivalent to showing that the set of $vw$-paths (directed
  paths from $v$ to $w$) in $O^\pi_{Y'}$ is the same as the set of
  $vw$-paths in $O^\sigma_{Y'}$. From this it will also follow that
  the diagram commutes. By assumption, both of these orientations
  contain at least one $vw$-path. We will consider separately the
  cases when these orientations share or do not share a common
  $vw$-path.
  
  \textit{Case 1: $O_{Y'}^\pi$ and $O_{Y'}^\sigma$ share no common
    $vw$-path.} Let $P_1$ be a length-$k_1$ $vw$-path in $O_{Y'}^\pi$,
  and let $P_2$ be a length-$k_2$ $vw$-path in
  $O_{Y'}^\sigma$. Suppose that in $O_{Y'}^\pi$ there are $k_2^+$
  edges along $P_2$ oriented from $v$ to $w$, and $k_2^-$ edges
  oriented from $w$ to $v$. Likewise, suppose that in $O_{Y'}^\sigma$
  there are $k_1^+$ edges along $P_1$ oriented from $v$ to $w$, and
  $k_1^-$ edges oriented from $w$ to $v$. If $C=P_1P_2^{-1}$ (the
  cycle formed by traversing $P_1$ followed by $P_2$ in reverse), then
  \begin{equation*}
    \nu_C(O_{Y'}^\pi)= k_1^++k_1^-+k_2^--k_2^+,\qquad
    \nu_C(O_{Y'}^\sigma)= k_1^+-k_1^--k_2^--k_2^+\;.
  \end{equation*}
  Equating these values yields $k_1^-+k_2^-=0$, and since these are
  non-negative integers, $k_1^-=k_2^-=0$. We conclude that $P_1$ is a
  $vw$-path in $O_{Y'}^\sigma$ and $P_2$ is a $vw$-path in
  $O_{Y'}^\pi$, contradicting the assumption that $O_{Y'}^\pi$ and
  $O_{Y'}^\sigma$ share no common $vw$-paths.

  \textit{Case 2: $O_{Y'}^\pi$ and $O_{Y'}^\sigma$ share a common
    $vw$-path $P_1$, say of length $k_1$}. If these are the only
  $vw$-paths, we are done. Otherwise, assume without loss of generality
  that $P_2$ is another $vw$-path in $O_{Y'}^\pi$, say of length
  $k_2$. Then if $C=P_1P_2^{-1}$, we have $\nu_C(O_{Y'}^\pi)=k_1-k_2$,
  and hence $\nu_C(O_{Y'}^\sigma)=k_1-k_2$. Therefore, $P_2$ is a
  $vw$-path in $O_{Y'}^\sigma$ as well. Because $P_2$ was arbitrary,
  we conclude that $O_{Y'}^\pi$ and $O_{Y'}^\sigma$ share the same set
  of $vw$-paths. Since Case 1 is impossible, we have established
  \eqref{eq:I_pi=I_sigma}, and the proof is complete. \qed
\end{proof}

Let $O_Y\in\Acyc(Y)$ and assume $I = \mathcal{I}(O_Y)$ has at least
two vertices. We write $Y_I$ for the graph formed from $Y$ by
contracting all vertices in $I$ to a single vertex denoted $V_I$. If
$I$ only contains $v$ and $w$ then $Y_I = Y''_e$. Moreover, $O_Y$
gives rise to an orientation $O_{Y_I}$ of $Y_I$, and this orientation
is clearly acyclic.

\begin{proposition}
  \label{prop:Y_I}
  Let $O^1_Y, O^2_Y \in \Acyc(Y)$ and assume $\mathcal{I}(O^1_Y) =
  \mathcal{I}(O^2_Y)$. If $O^1_Y \nkeq O^2_Y$ then $[O^1_{Y_I}] \nkeq
  [O^2_{Y_I}]$.
\end{proposition}

\begin{proof}
  We prove the contrapositive statement. Set $I=\mathcal{I}(O^1_Y)$,
  suppose $\card{I} = k$, and let $v_1v_2\cdots v_k$ be a linear
  extension of $I$. For any click-sequence $\mathbf{c}_I$ between two
  acyclic orientations $O^1_{Y_I}$ and $O^2_{Y_I}$ in $\Acyc(Y_I)$,
  let $\mathbf{c}$ be the click-sequence formed by replacing every
  occurrence of $c_i = V_I$ in $\mathbf{c}_I$ by the sequence
  $v_1\cdots v_k$. Then $\mathbf{c}(O^1_Y) = O^2_Y$ and $O^1_Y
  \tsimkappa O^2_Y$ as claimed. \qed
\end{proof}


\section{Proof of Theorem~\ref{thm:kappa}}
\label{sec:bijection}

In this section, we will utilize the results in the previous section
to establish a bijection from $\Acyc(Y)/\simkappa$ to
$\bigl(\Acyc(Y'_e)/\simkappa\cup\,\,\Acyc(Y''_e)/\simkappa\bigr)$ for
any cycle-edge $e$, which will in turn imply Theorem~\ref{thm:kappa}.

For $[O_Y]\in\Acyc(Y)/\simkappa$, let $O^\pi_Y$ denote an orientation
in $[O_Y]$ such that $\pi=v\pi_2\cdots\pi_n$ and $w=\pi_i$ for $i$
minimal. Define the map
\begin{equation}
  \Theta\colon\Acyc(Y)/\!\!\sim_\kappa\,\longrightarrow
  \big(\Acyc(Y')/\!\!\sim_\kappa\!\big)\,\bigcup\,
  \big(\Acyc(Y'')/\!\!\sim_\kappa\!\big)\; 
\end{equation}
by
\begin{equation}
  [O_Y]\stackrel{\Theta}{\longmapsto}
\begin{cases} 
  [O^\pi_{Y''}], &
  \mbox{$\exists O^{\pi}_Y\in[O_Y]$ with $\pi=vw\pi_3\cdots\pi_n$} \\ 
  [O^\pi_{Y'}], & \mbox{otherwise.} 
  \end{cases}
\end{equation}

Note that $[O_Y]$ is mapped into $\Acyc(Y'')/\simkappa$ if and only
if the only vertices in $\ieext([O_Y])$ are $v$ and $w$. Since
$\kappa$-equivalence over $Y$ implies $\kappa$-equivalence over $Y'$,
$\Theta$ does not depend on the choice of $\pi$, and is thus
well-defined. The results we have derived for the $vw$-interval now
allow us to establish the following:
\begin{theorem}
  \label{thm:bijection}
  The map $\Theta$ is a bijection. 
\end{theorem}
\begin{proof}
  We first prove that $\Theta$ is surjective. Let $I=\{v,w\}$ and
  consider an element $[O_{Y''}]\in\Acyc(Y'')/\!\sim_\kappa$ with
  $O^{\pi}_{Y''} \in [O_{Y''}]$ where $\pi =
  V_I\pi_2\cdots\pi_{n-1}$. Let $\pi^+=vw\pi_2\cdots\pi_{n-1} \in
  S_Y$.  Clearly $[O^{\pi^+}_Y] \in \Acyc(Y)/\!\sim_\kappa$ is mapped
  to $[O_{Y''}]$ by $\Theta$.
  
  Next, consider an element $[O_{Y'}] \in \Acyc(Y')/\!\sim_\kappa$. If
  there is no element $O^\pi_{Y'}$ of $[O_{Y'}]$ such that $\pi =
  vw\pi_3\cdots\pi_n$, then no elements of $[O_Y]$ are of this form
  either, and by definition $[O_{Y'}]$ has a preimage under
  $\Theta$. 
  We are left with the case where $[O_{Y'}]$ contains an element
  $O^\pi_{Y'}$ such that $\pi=vw\pi_3\cdots\pi_n$, and we must show
  that there exists $O^{\pi'}_{Y'}\in[O_{Y'}]$ such that
  $[O^{\pi'}_Y]$ contains no element of the form $O^\sigma_Y$ with
  $\sigma = vw\sigma_3\cdots\sigma_n$.
  Note that if $\sigma = vw\sigma_3\cdots\sigma_n$, then the vertices
  in $\mathcal{I}(O^\sigma_Y)$ are precisely $v$ and $w$. If the
  orientation $O_{Y'}$ had a directed path from $v$ to $w$, then the
  corresponding orientation $O_Y\in\Acyc(Y)$ formed by adding the edge
  $e$ with orientation $(v,w)$ has $vw$-interval of size at least 3,
  so by Proposition~\ref{prop:I*}, the acyclic orientation $O_Y$
  cannot be $\kappa$-equivalent to any orientation $O^\sigma_Y$ such
  that $\sigma=vw\sigma_3\cdots\sigma_n$.
  
  Thus it remains to consider the case when $[O_{Y'}]$ contains no
  acyclic orientation with a directed path from $v$ to $w$. Pick any
  simple undirected path $P'$ from $v$ to $w$ in $Y'$, which exists
  because $e$ is a cycle-edge. Choose an orientation in $[O_{Y'}]$ for
  which $\nu_{P'}$ is maximal. Without loss of generality we may
  assume that $O_{Y'}$ is this orientation.
  Let $O_Y\in\Acyc(Y)$ be the orientation that agrees with $O_{Y'}$,
  and with $e$ oriented as $(w,v)$. Since we have assumed that there
  is no directed path from $v$ to $w$ this orientation is acyclic. We
  claim that for any $\sigma=vw\sigma_3\cdots\sigma_n$ one has
  $O^\sigma_Y \not\in [O_Y]$. To see this, assume the statement is
  false. Let $P$ be the undirected cycle in $Y$ formed by adding the
  edge $e$ to the path $P'$. Because $e$ is oriented as $(v,w)$ in
  $O^\sigma_Y$ and as $(w,v)$ in $O_Y$, we have $\nu_P(O^\sigma_Y) =
  \nu_{P'}(O^\sigma_{Y'})-1$ and $\nu_P(O_Y)=\nu_{P'}(O_{Y'})+1$. If
  $O_Y$ and $O^\sigma_Y$ were $\kappa$-equivalent, then
  \begin{equation*}
    \nu_{P'}(O^\sigma_{Y'})-1=\nu_P(O^\sigma_Y)=\nu_P(O_Y)=\nu_{P'}(O_Y)+1\;,
  \end{equation*}
  and thus $\nu_{P'}(O^\sigma_{Y'})=\nu_{P'}(O_Y)+2$. Any click
  sequence mapping $O_Y$ to $O^\sigma_Y$ is a click-sequence from
  $O_{Y'}$ to $O^\sigma_{Y'}$. Therefore, $O^\sigma_{Y'} \in
  [O_{Y'}]$, which contradicts the maximality of $\nu_{P'}(O_{Y'})$. We
  therefore conclude that $O^\sigma_Y \not\in [O_Y]$, that
  $\Theta([O_Y]) = [O_{Y'}]$, and hence that $\Theta$ is surjective.
  
\medskip

  We next prove that $\Theta$ is an injection. By
  Proposition~\ref{prop:Y_I} (with $I=\{v,w\}$), $\Theta$ is injective
  when restricted to the preimage of $[O_{Y''}]$ under $\Theta$. Thus
  it suffices to show that every element in $\Acyc(Y')/\!\sim_\kappa$
  has a unique preimage under $\Theta$.
  By Proposition~\ref{prop:diagram}, every preimage of $[O_{Y'}]$ must
  have the same $vw$-interval $I$, containing $k>2$ vertices. Suppose
  there were preimages $[O_Y^\pi]\neq[O_Y^\sigma]$ of $[O_{Y'}]$.  By
  Proposition~\ref{prop:Y_I}, it follows that $O^\pi_{Y_I}\nkeq
  O^\sigma_{Y_I}$. We will now show that this leads to a
  contradiction.

  Assume that $\mathbf{c}=c_1\cdots c_m$ is a click-sequence from
  $O_{Y'}^\pi$ to $O_{Y'}^\sigma$. If one of $\pi$ or $\sigma$ is not
  $\kappa$-equivalent to a permutation with vertices $v$ and $w$ in
  succession, then their corresponding $\kappa$-classes would be
  unchanged by the removal of edge $e$. In light of this, we may
  assume that $\pi = v\pi_2\dots\pi_{n-1}w$ and $\sigma = v\sigma_2
  \dots \sigma_{n-1}w$, and thus that $c_1 = v$ and $c_m = w$. By
  Proposition~\ref{prop:forcedclick}, we may assume that the vertices
  in $I$ appear in $\mathbf{c}$ in some number of disjoint consecutive
  ``blocks,'' i.e., subsequences of the form $c_i\cdots c_{i+k-1}$.
  Replacing each of these blocks with $V_I$ yields a click-sequence
  from $O^\pi_{Y_I}$ to $O^\sigma_{Y_{I}}$, contradicting the fact
  that $O^\pi_{Y_I}\nkeq O^\sigma_{Y_I}$. Therefore, no such click
  sequence $\mathbf{c}$ exists, and $\Theta$ must be an injection, and
  the proof is complete. \qed
\end{proof}

Clearly, Theorem~\ref{thm:bijection} implies Theorem~\ref{thm:kappa}.
It is also interesting to note that the bijection $\beta_e \colon
\Acyc(Y) \longrightarrow \Acyc(Y'_e) \cup \Acyc(Y''_e)$
in~\eqref{eq:bijection2} does \emph{not} extend to a map on
$\kappa$-classes.


\section{Discrete Dynamical Systems, Coxeter Groups, Node-firing
  Games, and Quiver Representations}
\label{sec:summary}

We conclude with a brief discussion of how the equivalence relations
studied in this paper arise in various areas of mathematics. The
original motivation came from the authors' interest in
\emph{sequential dynamical systems} (\sdss). The equivalence relation
$\sim_Y$ arises naturally in the study of functional equivalence of
these systems. This can be seen as follows. Given a graph $Y$ with
vertex set $\{1,2,\ldots,n\}$ as above, a state $x_v\in K$ is assigned
to each vertex $v$ of $Y$ for some finite set $K$.  The \emph{system
state} is the tuple consisting of all the vertex states, and is
denoted $x = (x_1,\ldots,x_n)\in K^n$. The sequence of states
associated to the $1$-neighborhood $B_1(v;Y)$ of $v$ in $Y$ (in some
fixed order) is denoted $x[v]$. A sequence of vertex functions
$(f_i)_i$ with $f_i \colon K^{d(i)+1} \longrightarrow K$ induces
$Y$-local functions $F_i \colon K^n \longrightarrow K^n$ of the form
\begin{equation*}
  F_i(x_1,\ldots,x_n) = 
  (x_1,\ldots,x_{i-1}, f_i(x[i]), x_{i+1}, \ldots, x_n)\;.
\end{equation*}
The sequential dynamical system map with update order $\pi = (\pi_i)_i
\in S_Y$ is the function composition
\begin{equation}
  [\Fy,\pi] = F_{\pi(n)} \circ F_{\pi(n-1)} \circ \cdots \circ
  F_{\pi(2)} \circ F_{\pi(1)} \;.
\end{equation}
By construction, if $\pi\sim_Y\pi'$ holds then $[\Fy,\pi]$ and
$[\Fy,\pi']$ are identical as functions. Thus, $\acyc(Y)$ is an upper
bound for the number of functionally non-equivalent \sds maps that can
be generated over the graph $Y$ for a fixed sequence of vertex
functions. For any graph $Y$, there exist $Y$-local functions for
which this bound is sharp~\cite{Mortveit:01a}. A weaker form of
equivalence is \textit{cycle equivalence}, which means that the
dynamical system maps are conjugate (using the discrete topology) when
restricted to their sets of periodic points. In the language of graph
theory, this means their periodic orbits are isomorphic as directed
graphs.
The following result shows how $\kappa$- and $\delta$-equivalent
update orders yield dynamical system maps that are cycle equivalent.
\begin{theorem}
  \label{thm:shift}
  For any finite set $K$ of vertex states, and for any $\pi\in S_Y$,
  the \sds maps $[\Fy,\pi]$ and $[\Fy,\Shift_s(\pi)]$ are cycle
  equivalent. For vertex states $K=\F_2$ the \sds maps $[\Fy,\pi]$ and
  $[\Fy,\rev(\pi)]$ are cycle equivalent as well.
\end{theorem}

We refer to~\cite{Macauley:08d} for the proof of this result. Two
sequential dynamical systems are cycle equivalent if their phase space
digraphs are isomorphic when restricted to the cycles. The
paper~\cite{Macauley:08d} contains additional background on
equivalences of sequential dynamical systems as well as applications
of $\kappa$-equivalence to the structural properties of their phase
spaces.

\medskip

There is also a connection between $\kappa$-equivalence and the
theory of Coxeter groups. There is a natural bijection between the
set of Coxeter elements $\prod_i s_{\pi(i)}$ of a Coxeter group (see,
\eg~\cite{Bjorner:05}) with generators $s_i$ for $1\le i\le n$ and
Coxeter graph $Y$ (ignoring order labels), and the set of equivalence
classes $[\pi]_Y$. This is clear since the commuting generators are
precisely those that are not connected in $Y$. Let $C(W,S)$ denote the
set of Coxeter elements of a Coxeter group $W$ with generating set
$S=\{s_1,\dots,s_n\}$. It is shown in~\cite{Shi:97a} that there is a
bijection
\begin{equation}
  \label{eq:bij2}
  C(W,S) \longrightarrow \Acyc(Y) \;.
\end{equation}
Moreover, conjugating a Coxeter element $\prod s_{\pi(i)}$ by
$s_{\pi(1)}$ corresponds to a cyclic shift, \ie,
\begin{equation}
  s_{\pi(1)} (s_{\pi(1)} s_{\pi(2)} \cdots s_{\pi(n)}) s_{\pi(1)}
  = s_{\pi(2)} \cdots s_{\pi(n)} s_{\pi(1)}\;,
\end{equation}
since each generator $s_i$ is an involution.  Thus $\kappa(Y)$ is an
upper bound for the number of conjugacy classes of Coxeter elements of
a Coxeter group that has (unlabeled) Coxeter graph $Y$. This bound is
known to be sharp in certain cases~\cite{Shi:01}, but sharpness in the
general case is still an open question.

Click-sequences are closely related to $c$-admissible sequences of a
Coxeter element $c$. Recently, the structure of these sequences was
studied and used to prove that a power of a Coxeter element of an
infinite group is reduced~\cite{Speyer:07}.

\medskip

The \textit{chip-firing} game was introduced by Bj\"{o}rner,
Lov\'{a}sz, and Shor~\cite{Bjorner:91}. It is played over an
undirected graph $Y$, and each vertex is given some number (possibly
zero) of chips. If vertex $i$ has degree $d_i$, and at least $d_i$
chips, then a legal move (or a ``click'') of vertex $i$ is a
transfer of one chip to each neighboring vertex. This may be viewed as
a generalization of a source-to-sink move for acyclic orientations
where the out-degree of a vertex plays the role of the chip count. The
chip-firing game is closely related to the \textit{numbers
  game}~\cite{Bjorner:05}. In the numbers game over a graph $Y$, the
legal sequences of moves are in 1--1 correspondence with the reduced
words of the Coxeter group with Coxeter graph $Y$. For an excellent
summary and comparison of these games, see~\cite{KEriksson:94}.

\medskip

A quiver is a finite directed graph (loops and multiple edges are
allowed), and appears primarily in the study of representation
theory. A quiver $Q$ with a field $K$ gives rise to a path algebra
$KQ$, and there is a natural correspondence between $KQ$-modules, and
$K$-representations of $Q$. In fact, there is an equivalence between
the categories of quiver representations, and modules over path
algebras. A path algebra is finite-dimensional if and only if the
quiver is acyclic, and the modules over finite-dimensional path
algebras form a reflective subcategory. A \textit{reflection functor}
maps representations of a quiver $Q$ to representations of a quiver
$Q'$, where $Q'$ differs from $Q$ by a source-to-sink
operation~\cite{Marsh:03}. We note that while the composition of $n$
source-to-sink operations (one for each vertex) maps a quiver back to
itself, the corresponding composition of reflection functors is not
the identity, but rather a \textit{Coxeter functor}. In fact, the same
result in~\cite{Speyer:07} about powers of Coxeter elements being
reduced was proven previously using techniques from the representation
theory of quivers~\cite{Kleiner:07}.

\medskip

We conclude with a remark on the evaluation of the functions
$\alpha(Y)$ and $\kappa(Y)$. They both satisfy recurrences under edge
deletion and contraction, and may be computed through evaluations of
the Tutte polynomial $T_Y$ of $Y$. It is well-known that
$\alpha(Y)=T_Y(2,0)$, and in~\cite{Macauley:08b}, we showed that
$\kappa(Y)=T_Y(1,0)$. Other quantities counted by $T_Y(1,0)$ include
the number of acyclic orientations of $Y$ with a unique sink at a
fixed vertex~\cite{Gioan:07}, and the M\"obius invariant of the
intersection lattice of the graphic hyperplane arrangement of
$Y$~\cite{Novik:02}. Some of the results in this paper have a natural
interpretation in the language of the Tutte polynomial. For example,
Corollary~\ref{cor:bipartite} tells us that a connected graph $Y$ is
bipartite if and only if $T_Y(1,0)$ is odd. Proposition~\ref{prop:1/n}
implies that $n\cdot T_Y(1,0)\leq T_Y(2,0)$. We also point out a
previous study of the Tutte polynomial in the context of the
chip-firing game~\cite{Lopez:97}. We hope this paper will motivate
further explorations of the connections between these topics, as also
provide insights useful to some of the open problems in Coxeter
theory, in particular the sharpness of the bound $\kappa(Y)$ for the
enumeration of conjugacy classes of Coxeter elements.


\begin{acknowledgement}
  Both authors are grateful to the NDSSL group at Virginia Tech for
  the support of this research. Special thanks to Ed Green and Ken
  Millett for helpful discussions and feedback, and to William
  Y.~C. Chen for valuable advise regarding the preparation and
  structuring of this paper. The work was partially supported by
  Fields Institute in Toronto, Canada.
\end{acknowledgement}


\end{document}